\RequirePackage{fix-cm}
\documentclass[smallextended]{svjour3}       
\smartqed  
\usepackage[unicode,colorlinks=true,linkcolor=blue,urlcolor=blue,citecolor=blue]{hyperref}%
\usepackage{graphicx}
\usepackage{amsmath,amssymb,enumitem,booktabs}
\spnewtheorem{myproposition}[theorem]{Proposition}{\bf}{\it}
\spnewtheorem{mycorollary}[theorem]{Corollary}{\bf}{\it}
\spnewtheorem{myclaim}{Claim}{\it}{\rm}

\newlist{numberedlist}{enumerate}{2}
\setlist[numberedlist]{label=(\arabic{section}.\arabic*),leftmargin=*}

\journalname{Mathematical Programming}

\renewcommand{\makeheadbox}{{%
\hbox to0pt{\vbox{\baselineskip=10dd\hrule\hbox
to\hsize{\vrule\kern3pt\vbox{\kern3pt
\hbox{\bfseries Mathematical Programming (2019)}
\hbox{\href{https://doi.org/10.1007/s10107-019-01427-7}{DOI
    10.1007/s10107-019-01427-7}}
\kern3pt}\hfil\kern3pt\vrule}\hrule}%
\hss}}}

\begin{document}

\title{New Polyhedral and Algorithmic Results on Greedoids \thanks{The
    research reported in this paper has been supported by the National
    Research, Development and Innovation Fund (TUDFO/51757/2019-ITM,
    Thematic Excellence Program). The research reported in this paper
    was supported by the BME - Artificial Intelligence FIKP grant of
    EMMI (BME FIKP-MI/SC). Research supported by No. OTKA 185101 of
    the Hungarian Scientific Research Fund.}}



\author{D\'avid Szeszl\'er}


\institute{D. Szeszl\'er \at
Department of Computer Science and Information Theory\\
Budapest University of Technology and Economics\\
Magyar Tud\'osok K\"or\'utja 2., Budapest, 1117, Hungary\\
\email{szeszler@cs.bme.hu}           
}

\date{Published online: 26 August 2019}

\maketitle

\begin{abstract}
We present various new results on greedoids.  We prove a theorem that
generalizes an equivalent formulation of Edmonds' classic matroid
polytope theorem to local forest greedoids -- a class of greedoids
that contains matroids as well as branching greedoids. We also
describe an application of this theorem in the field of measuring the
reliability of networks by game-theoretical tools. Finally, we prove
new results on the optimality of the greedy algorithm on greedoids and
correct some mistakes that have been present in the literature for
almost three decades.
\end{abstract}

\section{Introduction}\label{sect:intro}

Greedoids were introduced by Korte and Lov\'asz at the beginning of
the 1980s as a generalization of matroids. The motivation behind the
concept was the observation that in the proofs of various results on
matroids subclusiveness (that is, the property that all subsets of
independent sets are also independent) is not needed.  Besides
matroids, the class of greedoids includes some further very important
combinatorial objects such as the edge sets of subtrees of a graph
rooted at a given node. 

Although the research of greedoids was very active until the
mid-1990s, the topic seems to have faded away since then. Most of the
known results on greedoids are already included in the comprehensive
book of Korte, Lov\'asz and Schrader \cite{KLS} published in 1991.
The fact that greedoids have not gained as much importance within
combinatorial optimization as matroids is probably due to the fact
that the class of greedoids is much more diverse than that of matroids
and classic concepts and results on matroids do not seem to generalize
easily to greedoids.

The motivations behind the results of this paper are
threefold. Firstly, we identify a class of greedoids, \emph{local
  forest greedoids}, that includes both matroids and branching
greedoids and that admits a generalization of a fundamental polyhedral
result on matroids: Edmonds' classic theorem on the polytope spanned
by incidence vectors of independent sets of a matroid. In particular,
we prove a generalization of an eqivalent formulation of this theorem
to local forest greedoids.  To the best of our knowledge, no
generalization of (any form of) the matroid polytope theorem to
greedoids has been known. We do this partly in the (perhaps vain) hope
that further fundamental results on matroids will turn out to be
generalizable to this class of greedoids.

Secondly, we aim at generalizing some results obtained in
\cite{Sz}. There we considered some attacker-defender games played on
graphs with the aim of defining new security metrics of graphs and
better understanding others that had been known in the literature. For
this purpose, we defined a general framework involving matroids: the
\emph{Matroid Base Game} is a two-player, zero sum game in which the
Attacker aims at hitting a base chosen by the Defender. In particular,
the Attacker chooses an element of the ground set of a given matroid
and the Defender chooses a base of the same matroid; then the payoff
depends on both of their choices in such a way that it is favorable
for the Attacker if his chosen element belongs to the base chosen by
the Defender. The results of \cite{Sz} on the Matroid Base Game served
as a common generalization of some results that had been known in the
literature on measuring the security of networks via game-theoretical
means. In particular, the Nash-equilibrium payoff of the Matroid Base
Game was determined and it was proved that it is a common
generalization of some known graph reliability metrics. However, there
are other known metrics of a very similar nature which did not fit
into the framework provided by the Matroid Base Game. In this paper we
further generalize the definition of the Matroid Base Game by
replacing matroids with local forest greedoids and we prove that some
of the results of \cite{Sz} generalize to this case too. We also show
that this more general framework is capable of handling and
generalizing some further graph reliability metrics known from the
literature beyond the ones already contained in the framework provided
by the matroid base game.

The third motivation behind the results of this paper is to better
understand the conditions under which the greedy algorithm is optimal
on greedoids. This question is a central topic in the literature of
greedoids, the name greedoid itself comes from ``a synthetic blending
of the words greedy and matroid'' \cite{KLS} which indicates that one
of the basic motivations of the notion was to extend the theoretical
background behind greedy algorithms beyond the well-known results on
matroids. Accordingly, Korte and Lov\'asz proved some fundamental
results on the optimality of the greedy algorithm on greedoids in
\cite{KL1} and \cite{KL2} which were also presented and further
extended in \cite{KLS}. Most surprisingly however, they seem to have
overlooked a detail which led them to some false claims. These
mistakes, which seem to have remained hidden in the literature of
greedoids for almost three decades, will be pointed out and
corrections will be proposed and proved.

It should be emphasized though that, although the optimality of the
greedy algorithm in a certain special case (see
Theorem~\ref{thm:boyd}) will be a crucial tool for proving the above
mentioned polyhedral result, the results of the present paper on the
optimality of the greedy algorithm are not needed for this proof, it
is only Theorem~\ref{thm:boyd} proved in \cite{Boyd} that is relied
on.

This paper is structured as follows. In Sections \ref{sect:greedoid}
and \ref{sect:prelimgreedy} all the necessary background on greedoids
and the greedy algorithm on greedoids, respectively, is given.  In
Section~\ref{sect:mainresult} the above mentioned polyhedral result is
clamied and proved and in Section~\ref{sect:apply} we briefly outline
an application of this result concerning the measurement of the
reliability of networks. Finally, Section~\ref{sect:greedyalg} is
dedicated to some results on the optimality of the greedy algorithm on
greedoids.

\section{Preliminaries on Greedoids}\label{sect:greedoid}

All the definitions and claims in this section are taken from
\cite{KLS}.

A \emph{greedoid} $G=(S,\mathcal{F})$ is a pair consisting of a finite
ground set $S$ and a collection of its subsets $\mathcal{F}\subseteq
2^S$ such that the following properties are fulfilled:
\begin{numberedlist}
\item $\emptyset\in\mathcal{F}$
\item If $X,Y\in\mathcal{F}$ and $|X|<|Y|$ then there exists a 
$y\in Y-X$ such that $X+y\in\mathcal{F}$.\label{prop:exchange}
\end{numberedlist}

When we apply \ref{prop:exchange} on the sets $X,Y\in\mathcal{F}$ with
$|X|<|Y|$, we say that we \emph{augment} $X$ from $Y$.
Members of $\mathcal{F}$ are called \emph{feasible sets}. Obviously,
the definition of greedoids is obtained from that of matroids by
relaxing subclusiveness, that is, subsets of feasible sets are not
required to be feasible any more. On the other hand,
\ref{prop:exchange} immediately implies that every $X\in\mathcal{F}$
has a \emph{feasible ordering}: $(x_1,x_2,\ldots,x_k)$ is a feasible
ordering of $X$ if $X=\{x_1,x_2,\ldots,x_k\}$ and
$\{x_1,x_2,\ldots,x_i\}\in\mathcal{F}$ holds for every $1\le i\le
k$. The existence of a feasible ordering, in turn, implies the
\emph{accessible property} of greedoids: for every $\emptyset\ne
X\in\mathcal{F}$ there exists an $x\in X$ such that
$X-x\in\mathcal{F}$.

In this paper, the following notations will be (and have been) used:
for a subset $X\subseteq S$ and an element $x\in S$ we will write
$X+x$ and $X-x$ instead of $X\cup\{x\}$ and $X-\{x\}$,
respectively. Furthermore, given any function
$c:S\rightarrow\mathbb{R}$ and a subset $X\subseteq S$, $c(X)$ will
stand for $\sum\{c(x):x\in X\}$.

Some of the well-known terminology on matroids can be applied to
greedoids without any modification. In particular, the \emph{rank}
$r(A)$ of a set $A\subseteq S$ is $r(A)=\max\{|X|:X\subseteq A,
X\in\mathcal{F}\}$. Given a subset $A\subseteq S$, a \emph{base} of
$A$ is a subset $X\subseteq A$, $X\in\mathcal{F}$ of maximum
size. This, by property~\ref{prop:exchange} is equivalent to saying
that $X+y\notin\mathcal{F}$ for every $y\in A-X$. A base of $S$ is
called a base of the greedoid $G=(S,\mathcal{F})$. The set of bases of
$G$ will be denoted by $\mathcal{B}$.

\emph{Minors of greedoids} can also be defined almost identically to
those of matroids. If $G=(S,\mathcal{F})$ is a greedoid and
$X\subseteq S$ is an arbitrary subset then the \emph{deletion} of $X$
yields the greedoid $G\setminus X=(S-X,\mathcal{F}\setminus X)$, where
$\mathcal{F}\setminus X=\{Y\subseteq S-X:Y\in\mathcal{F}\}$. If
$X\in\mathcal{F}$ is a feasible set then the \emph{contraction} of $X$
yields the greedoid $G/X=(S-X,\mathcal{F}/X)$, where $\mathcal{F}/X=\{
Y\subseteq S-X: Y\cup X\in\mathcal{F}\}$. Then a minor of $G$ is
obtained by applying these two operations on $G$. It is
straightforward to check that minors are indeed greedoids. (Note,
however, that $G/X$ was only defined here in the $X\in\mathcal{F}$
case. The definition could be extended to a wider class of subsets,
but unless some further structural properties are imposed on the
greedoid, not to arbitrary ones. See \cite[Chapter V.]{KLS} for the
details.)

In this paper the following terminology will also be used:
$X\subseteq S$ will be called \emph{subfeasible} if there exists
a $Y\in\mathcal{F}$ such that $X\subseteq Y$. The set of subfeasible
sets will be denoted by $\mathcal{F}^{\vee}$.

There are many known examples of greedoids beyond matroids and they
arise in diverse areas of mathematics, see \cite{KLS} for an extensive
list. For the purposes of this paper, \emph{branching greedoids} will
be of importance. Let $H=(V,E_u,E_d)$ be a mixed graph (that is, it
can contain both directed and undirected edges) with $V$, $E_u$ and
$E_d$ being its set of nodes, undirected edges and directed edges,
respectively. Furthermore, let $r\in V$ be a given root node. The
ground set of the branching greedoid on $H$ is $E_u\cup E_d$ and
$\mathcal{F}$ consists of all subsets $A\subseteq E_u\cup E_d$ such
that disregarding the directions of the arcs in $A\cap E_d$, $A$ is
the edge set of a tree containing $r$ and for every path $P$
in $A$ starting in $r$ all edges of $P\cap E_d$ are directed away
from $r$. It is straightforward to check that $G=(E_u\cup
E_d,\mathcal{F})$ is indeed a greedoid. $G$ is called an
\emph{undirected branching greedoid} or a \emph{directed branching
  greedoid} if $H$ is an undirected graph (that is, $E_d=\emptyset$)
or a directed graph (that is, $E_u=\emptyset$), respectively.

Most of the known results on greedoids are about special classes of
greedoids, that is, further structural properties are assumed. Among
these, the following will be of relevance in this paper:
\begin{numberedlist}[start=3]
\item \emph{Local Union Property:}\\
if $A,B\in\mathcal{F}$ and $A\cup B\in\mathcal{F}^{\vee}$ then $A\cup
B\in\mathcal{F}$\label{lup}
\item \emph{Local Intersection Property:}\\
if $A,B\in\mathcal{F}$ and $A\cup B\in\mathcal{F}^{\vee}$ then $A\cap
B\in\mathcal{F}$\label{lip}
\item \emph{Local Forest Property:}\\
if $A, A+x, A+y, A\cup\{x,y\}, A\cup\{x,y,z\}\in\mathcal{F}$ then
either $A\cup\{x,z\}\in\mathcal{F}$ or $A\cup\{y,z\}\in\mathcal{F}$\label{lfp}
\end{numberedlist}

A greedoid $G=(S,\mathcal{F})$ is called an \emph{interval greedoid}
if it fulfills property~\ref{lup}; $G$ is a \emph{local poset
greedoid} if it fulfills \ref{lup} and \ref{lip}; finally,
$G$ is a \emph{local forest greedoid} if it fulfills \ref{lup},
\ref{lip} and \ref{lfp}. 

Obviously, all matroids are local forest greedoids and it is easy to
check that so are branching greedoids. However, there are further
examples that do not belong to either of these classes: for example,
the direct sum of the uniform matroid $U_{3,2}$ and a branching
greedoid that is not a matroid is also a local forest greedoid but it
is neither a matroid nor a branching greedoid (since all these
classes are closed under taking minors and $U_{3,2}$ is clearly not a
branching greedoid). Another type of example can be obtained from any
local forest greedoid (even a matroid): let $G=(S,\mathcal{F})$ be a
local forest greedoid, $X\in\mathcal{F}$ a feasible set and
$(x_1,x_2,\ldots,x_k)$ an arbitrary (not necessarily feasible)
ordering of $X$; then 
\[\mathcal{F}^{(x_1,\ldots,x_k)}=\{Y\in\mathcal{F}: X\subseteq Y\}
\cup\left\{\{x_1,x_2,\ldots,x_i\}:i=1,\ldots,k\right\}\cup
\{\emptyset\}\]
is also a local forest greedoid on $S$. 

Observe that in interval greedoids \ref{lup} implies that every
$X\in\mathcal{F}^{\vee}$ has a unique base; indeed, it is the union of
all feasible sets in $X$. In this paper, this unique base will be
denoted by $\Delta(X)$.  Analogously, \ref{lip} implies that if
$X\subseteq A$ and $A\in\mathcal{F}$ then there is a unique minimum
size feasible set containing $X$ in $A$. This gives rise to the
definition of paths: if $G=(S,\mathcal{F})$ is a local poset greedoid,
$A\in\mathcal{F}$ and $x\in A$ then the \emph{$x$-path} in $A$,
denoted by $P_x^A$ (or simply $P_x$ if this is unambiguous) is the
unique feasible set in $A$ containing $x$ such that no proper feasible
subset of $P_x^A$ contains $x$. Clearly, in case of branching
greedoids this notion translates to paths starting in the root node
$r$ that are directed in the sense that all directed edges in the path
are directed away from $r$. The following theorem was proved in
\cite{Schmidt}; we also give a simple proof here for the sake of
self-containedness.

\begin{theorem}[W. Schmidt, 1988 \cite{Schmidt},{\cite[Theorem
    VII.4.4]{KLS}}]\label{thm:lfgpath}
Let $G=(S,\mathcal{F})$ be a local poset greedoid. Then the following
are equivalent:
\begin{enumerate}[label=(\roman*)]
\item $G$ is a local forest greedoid (that is, it fulfills \ref{lfp});
\item every path in $G$ has a unique feasible ordering;
\item if $(a_1,a_2,\ldots,a_k)$ is the feasible ordering of a path
  then $\{a_1,a_2,\ldots,a_i\}$ is also a path for every $1\le i\le k$.
\end{enumerate}
\end{theorem}
\begin{proof}
Assume by way of contradiction that (i) is fulfilled but (ii) is not
and choose a path $P_z$ of minimum cardinality that has two different
feasible orderings: $(a_1,\ldots,a_k,x,z)$ and
$(b_1,\ldots,b_k,y,z)$. We first show that $x\ne y$ can be assumed
without loss of generality. Indeed, if $x=y$ then
$\{a_1,\ldots,a_k,x\}$ is not a path by the minimality of $|P_z|$, but
since it is feasible, it contains an $x$-path $P_x$ as a proper
subset. Then augmenting a feasible ordering of $P_x$ from 
$P_z$ by \ref{prop:exchange} we get a feasible ordering of $P_z$ the
second to last element of which is not $y$. So assume $x\ne y$ 
and let $A=P_z-\{x,y,z\}$. Then clearly $A+x\in\mathcal{F}$ and
$A+y\in\mathcal{F}$ hold by the feasibility of the two orderings and 
$A\cup\{x,y,z\}=P_z\in\mathcal{F}$. Hence by \ref{lup} and
\ref{lip} we have $A\cup\{x,y\}\in\mathcal{F}$ and $A\in\mathcal{F}$
too. Therefore \ref{lfp} implies $A\cup\{x,z\}\in\mathcal{F}$ or 
$A\cup\{y,z\}\in\mathcal{F}$. In both cases we get a smaller feasible
set containing $z$ than $P_z$ contradicting the definition of $P_z$.

Proving (iii) from (ii) is almost immediate: if
$\{a_1,a_2,\ldots,a_i\}$ were not a path then it would contain an
$a_i$-path by definition which could be augmented by
\ref{prop:exchange} from $P=\{a_1,\ldots,a_k\}$ to obtain a different
feasible ordering of $P$.

Finally, we show (i) from (iii).  Let $A$ and $x,y,z$ be given
according to $\ref{lfp}$ and let $B=A\cup\{x,y,z\}$.  Clearly, if
$A\cap\{x,y,z\}\ne\emptyset$ or $|\{x,y,z\}|<3$ then \ref{lfp} is
automatically fulfilled, so we can assume that neither of these is the
case.  Since $A+x\in\mathcal{F}$, we have $P_x^B\subseteq A+x$ and
hence $y\notin P_x^B$. Similarly, $x\notin P_y^B$. This, by (iii),
implies that $P_z^B$ contains at most one of $x$ and $y$; indeed, if
it contained both and, for example, $x$ preceded $y$ in a feasible
ordering of $P_z^B$ then since the prefix of this ordering up to $y$
would be a path by (iii), we would get $x\in P_y^B$. So assume
$y\notin P_z^B$ without loss of generality. Then applying \ref{lup} on
$A+x$ and $P_z^B$, both of which are subsets of $B\in\mathcal{F}$, we
get $A\cup\{x,z\}\in\mathcal{F}$ as claimed.\qed
\end{proof}

\section{Preliminaries on the Greedy Algorithm in
  Greedoids}\label{sect:prelimgreedy}

As mentioned in the Introduction, the notion of greedoids was
motivated by the fact that they provide the underlying structure for 
a simple greedy algorithm.

Let $G=(S,\mathcal{F})$ be an arbitrary greedoid and
$w:\mathcal{F}\rightarrow\mathbb{R}$ an objective function.  Assume
that we are interested in finding a base $B\in\mathcal{B}$ that
maximizes $w(B)$ across all bases of $G$.  For every $A\in\mathcal{F}$
the \emph{set of continuations} of $A$ is defined as $\Gamma(A)=\{x\in
S-A: A+x\in\mathcal{F}\}$. Then the \emph{greedy algorithm} for the
above problem can be described as follows \cite{KL1,KLS}:

\begin{enumerate}[labelindent=\parindent,itemsep=0pt,label=\textit{Step \arabic*.},leftmargin=*]
\item Set $A=\emptyset$.

\item If $\Gamma(A)=\emptyset$ then stop and
output $A$. 

\item Choose an $x\in\Gamma(A)$ such that
$w(A+x)\ge w(A+y)$ for every $y\in\Gamma(A)$.

\item Replace $A$ by $A+x$ and continue at
\textit{Step 2}.
\end{enumerate}

Obviously, if one is interested in minimizing $w(B)$ across all bases
then, since this is equivalent to maximizing $-w(B)$, the only
modification needed in the algorithm is to require $w(A+x)\le w(A+y)$
for every $y\in\Gamma(A)$ in Step 3.

Many of the well-known, elementary algorithms in graph theory fall
under this framework as shown by the following examples.

\begin{example}\label{ex:matroid}
If $M$ is a matroid and $w$ is linear (meaning that $w(A)=c(A)$ for
some weight function $c:S\rightarrow\mathbb{R}$) then the above greedy
algorithm is nothing but the well-known greedy algorithm on
matroids. In particular, we get Kruskal's algorithm for finding a
maximum weight spanning tree in case of the cycle matroid.
\end{example}

\begin{example}\label{ex:prim}
Let $G$ be the branching greedoid of the undirected graph $H$ and $w$
a linear objective function. Then the greedy algorithm translates to
Prim's well-known algorithm for finding a maximum weight spanning
tree. (Note that this algorithm cannot be interpreted in a
matroid-theoretical context.)
\end{example}

\begin{example}\label{ex:dijkstra}
Let $G$ be the branching greedoid of the mixed graph $H=(V,E_u,E_d)$
with root node $r$ and let $c:E_u\cup E_d\rightarrow\mathbb{R}^+$ be a
non-negative valued weight function. Then let
$w(A)=\sum\{c(P_e^A):e\in A\}$ for every $A\in\mathcal{F}$. Korte and
Lov\'asz observed \cite{KL1} that in this case the greedy algorithm
for minimizing $w(B)$ translates to Dijkstra's well-known shortest
path algorithm. Indeed, Dijkstra's algorithm constructs a spanning
tree on the set of nodes reachable from $r$ such that the unique path
from $r$ to every other node in this tree is a shortest path and hence
it clearly minimizes $w$.
\end{example}

Although the greedy algorithm finds an optimum base in the above
examples, it is obviously not to be expected that this is true in
general. The first sufficient condition for the optimality of the
greedy algorithm was given by Korte and Lov\'asz in \cite{KL1}. There
they introduced an even broader framework: they considered objective
functions defined on all feasible orderings of feasible sets. Given a
greedoid $G=(S,\mathcal{F})$, let $\mathcal{L}(\mathcal{F})$ denote the
set of all feasible orderings of all feasible sets. Extending the greedy
algorithm to the case of an objective function
$w:\mathcal{L}(\mathcal{F})\rightarrow\mathbb{R}$ is obvious: instead
of augmenting a feasible set $A\in\mathcal{F}$, it keeps maintaining
and updating a feasible ordering of $A$ that is always augmented by the
best possible choice $x$.

\begin{theorem}[B. Korte and L. Lov\'asz, 1984
\cite{KL1},{\cite[Theorem XI.1.3]{KLS}}]\label{thm:kl}
Let $G=(S,\mathcal{F})$ be an arbitrary greedoid and 
$w:\mathcal{L}(\mathcal{F})\rightarrow\mathbb{R}$ an objective
function. Assume that whenever $(a_1,\ldots,a_i,x)$ is a feasible
ordering of a set
$A+x\in\mathcal{F}$ (where $i=0$ is possible) such that 
$w\left((a_1,\ldots,a_i,x)\right)\ge w\left((a_1,\ldots,a_i,y)\right)$
for every $y\in\Gamma(A)$ then the following conditions hold:
\begin{numberedlist}
\item\label{prop:kltetel1}
  $w\left((a_1,\ldots,a_i,b_1,\ldots,b_j,x,c_1,\ldots,c_k)\right)\ge$

\hfill%
$w\left((a_1,\ldots,a_i,b_1,\ldots,b_j,z,c_1,\ldots,c_k)\right)$

if both of these strings are in $\mathcal{L}(\mathcal{F})$ (and $j=0$
or $k=0$ is possible).
\item\label{prop:kltetel2}
  $w\left((a_1,\ldots,a_i,x,b_1,\ldots,b_j,z,c_1,\ldots,c_k)\right)\ge$ 

\hfill%
$w\left((a_1,\ldots,a_i,z,b_1,\ldots,b_j,x,c_1,\ldots,c_k)\right)$

if both of these strings are in $\mathcal{L}(\mathcal{F})$ (and $j=0$
or $k=0$ is possible).
\end{numberedlist}
Then the greedy algorithm finds a maximum base with respect to $w$.
\end{theorem}

Since in most applications the objective function only depends on the
feasible sets themselves and not on their orderings, one would want to
formulate the corresponding corollary of
Theorem~\ref{thm:kl}. Obviously, \ref{prop:kltetel2} is automatically
fulfilled in these cases, however, it is not at all straightforward to
specialize \ref{prop:kltetel1} to such objective functions.  Both in
\cite{KL1} and \cite[Chapter XI, condition (1.4)]{KLS} it is claimed
that for objective functions $w:\mathcal{F}\rightarrow\mathbb{R}$
\ref{prop:kltetel1} is equivalent to the following:
\begin{numberedlist}[start=3]
\item \label{prop:klrossz}
  If $A,B,A+x,B+x\in\mathcal{F}$ hold for some sets $A\subseteq B$
  and $x\in S-B$, and $w(A+x)\ge w(A+y)$ for every $y\in\Gamma(A)$ then 
  $w(B+x)\ge w(B+z)$ for every $z\in\Gamma(B)$.
\end{numberedlist}
This reformulation, however, clearly disregards the fact that
$\{a_1,\ldots,a_i,$ $b_1,\ldots,$ $b_j,c_1,\ldots,c_k\}$ need not be a
feasible set. In actual fact, \ref{prop:klrossz} does not guarantee
the optimality of the greedy algorithm as shown by the trivial example
of Figure~\ref{figure:1}: consider the undirected branching greedoid
of the graph on the left hand side and let the objective function be
defined as in the table on the right hand side. It is easy to check
that \ref{prop:klrossz} is fulfilled, however, the greedy algorithm
gives $\{a,c\}$ instead of $\{b,c\}$. On the other hand,
\ref{prop:kltetel1} is clearly violated: $a$ is the best continuation
of $\emptyset$ but $w((a,c))<w((b,c))$.

\addtolength{\arraycolsep}{3pt}

\begin{figure}[h]
\noindent\hfill\raisebox{-18pt}{\scalebox{.8}{\includegraphics{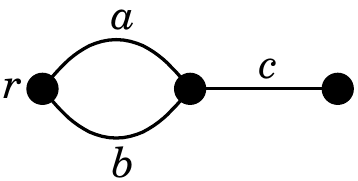}}}\hfill%
$\begin{array}{cccccc}\toprule
X & \emptyset & \{a\} & \{b\} & \{a,c\} & \{b,c\}\\\midrule
w(X) & 0 & 2 & 1 & 3 & 4\\\bottomrule
\end{array}$\hfill\ 
\caption{Property~\ref{prop:klrossz} does not imply
  property~\ref{prop:kltetel1} for order-independent objective
  functions, nor does it guarantee the optimality of the greedy
  algorithm} 
\label{figure:1}
\end{figure}

Unfortunately, as innocuous as the above mistake might look, it led
the authors of \cite{KLS} to the following false claim (see \cite[page
  156]{KLS}): if $G=(S,\mathcal{F})$ is a local poset greedoid and 
$w:\mathcal{F}\rightarrow\mathbb{R}$ is defined as
$w(A)=\sum\{c(P_x^A):x\in A\}$ for a $c:S\rightarrow\mathbb{R}^+$
analogously to Example~\ref{ex:dijkstra},
then the greedy algorithm finds a minimum base with respect to
$w$. To disprove this, let $S=\{x,y,z,u\}$, 
$\mathcal{F}=\left\{\emptyset, \{x\}, \{y\}, \{x,y\}, \{x,u\},
\{x,y,z\}, \{x,z,u\}\right\}$ and $c(x)=3$, $c(y)=2$,
$c(z)=c(u)=0$. Then it is easy to check that $(S,\mathcal{F})$ is a
local poset greedoid, but since the greedy algorithm starts with
choosing $y$, it terminates with $\{x,y,z\}$ which is not minimum as
$w(\{x,y,z\})=10$ and $w(\{x,z,u\})=9$. 

Moreover, it is worth noting that while the optimality of Dijkstra's
algorithm does follow from Theorem~\ref{thm:kl} for directed graphs,
it does not follow in the undirected case as shown by the example of
Figure~\ref{figure:2}: although $x$ is the best continuation of the
empty set, $11=w((x,b,a))>w((z,b,a))=10$, hence $(-w)$ violates
\ref{prop:kltetel1}. 

\begin{figure}[h]
\noindent\hfill\raisebox{-18pt}{\scalebox{.8}{\includegraphics{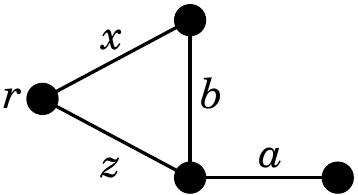}}}\hfill%
$\begin{array}{ccccc}\toprule
e & x & z & a & b\\\midrule
c(e) & 1 & 2 & 0 & 4\\\bottomrule
\end{array}$\hfill\ 
\caption{The optimality of Dijkstra's algorithm does not follow from
  Theorem~\ref{thm:kl} for undirected graphs}
\label{figure:2}
\end{figure}

\addtolength{\arraycolsep}{-3pt}

On the other hand, the following was shown in \cite{Boyd}.

\begin{theorem}[E. A. Boyd, 1988 \cite{Boyd}]\label{thm:boyd}
Let $G=(S,\mathcal{F})$ be a local forest greedoid,
$c:S\rightarrow\mathbb{R}^+$ a non-negative valued weight function and 
$w(A)=\sum\{c(P_x^A):x\in A\}$ for every feasible set
$A\in\mathcal{F}$. Then the greedy algorithm finds a minimum base with
respect to $w$.
\end{theorem}

Since both undirected and directed branching greedoids are local
forest greedoids, the above theorem implies the optimality of
Dijkstra's algorithm both for undirected and directed graphs.  A
generalization of Theorem~\ref{thm:boyd} will be given in
Section~\ref{sect:greedyalg} (see Theorem~\ref{thm:dijkstra}) the
proof of which will also be shorter than the rather technical one
given in \cite{Boyd}.

Although the optimality of the greedy algorithm is a central topic in
the theory of greedois, most results regarding this question are about
linear objective functions. In \cite{KL2} Korte and Lov\'asz proved
that on an arbitrary greedoid $G=(S,\mathcal{F})$ the greedy algorithm
is optimal for all linear objective functions if and only if the
following \emph{strong exchange axiom} is fulfilled: for every
$A\subseteq B$, $A\in\mathcal{F}$, $A+x\in\mathcal{F}$,
$B\in\mathcal{B}$ and $x\in S-B$ there exists a $y\in B-A$ such that
$B-y+x\in\mathcal{B}$ and $A+y\in\mathcal{F}$. A generalization of
this result to arbitrary objective functions will be given in
Section~\ref{sect:greedyalg} (see Theorem~\ref{thm:greedyopt}).  In
\cite{HMS} another generalization of the above result of Korte and
Lov\'asz \cite{KL2} was given: a necessary and sufficient condition for
the optimality of the greedy algorithm for linear objective functions
on accessible set systems. In \cite{Mao} a variant of the greedy
algorithm on interval greedoids that ``looks two step ahead'' is
defined and a necessary and sufficient condition for its optimality on
linear objective functions is derived. As for general (that is, not
necessarily linear) and possibly order-dependent objective functions a
generalization of Theorem~\ref{thm:kl} was most recently given in
\cite{Sz3} that, among other applications, completely covers
Example~\ref{ex:dijkstra} (also for undirected graphs).

\section{A Polyhedral Result}\label{sect:mainresult}

In this section we prove a generalization of Edmonds' classic matroid
polytope theorem to local forest greedoids. 

\begin{theorem}[J. Edmonds, 1971 \cite{Edmonds}]\label{thm:edmonds}
Let $M=(S,\mathcal{F})$ be a matroid with rank function
$r$ and let $P_{\text{ind}}(M)$ denote the polytope 
spanned by the incidence vectors of all independent sets of $M$. Then
\[\displaystyle
P_{\text{ind}}(M)=\left\{x\in\mathbb{R}^S:
x(U)\le r(U)\mbox{ for all }U\subseteq S, 
x(s)\ge 0\mbox{ for all }s\in S\right\}.\]
\end{theorem}

The theorem has some equivalent formulations, the one of relevance for
the purposes of this paper is the following. The \emph{up-hull} of a
polyhedron $P\subseteq\mathbb{R}^n$, denoted by $P^{\uparrow}$, is
defined as $P^{\uparrow}=\{x\in\mathbb{R}^n:\exists z\in P, z\le
x\}$; in other words, $P^{\uparrow}$ is the Minkowski-sum of $P$ and
the non-negative orthant of $\mathbb{R}^n$ (and as such, it is also a
polyhedron). 

\begin{theorem}\label{thm:matroidbaseuphull}
Let $M=(S,\mathcal{F})$ be a matroid with rank function
$r$ and let $P_{\text{base}}(M)$ denote the polytope 
spanned by the incidence vectors of all bases of $M$. Then
\[\displaystyle
P_{\text{base}}^{\uparrow}(M)=\left\{x\in\mathbb{R}^S:
x(U)\ge r(S)-r(S-U)\mbox{ for all }U\subseteq
S\right\}.\]
\end{theorem}

As claimed above, this theorem is just a reformulation of
Theorem~\ref{thm:edmonds}. Indeed, by applying
Theorem~\ref{thm:edmonds} to the dual of a matroid one gets a
description of the polytope spanned by the incidence vectors of all
spanning sets (that is, sets containing a base of $M$); then it is
easy to check that this polytope is nothing but the intersection of
$P_{\text{base}}^{\uparrow}(M)$ and the hypercube $[0,1]^S$. The
details are given in \cite[Chapter 40.2]{Schrijver}.

\begin{definition}\label{dfn:shadow}
Given a greedoid $G=(S,\mathcal{F})$, a feasible set $A\in\mathcal{F}$
and an $x\in S$, the \emph{shadow} of $x$ on $A$ is
$sh_A(x)=|A|-r(A-x)$. The \emph{shadow vector} of $A$ is
the vector $sh_A\in\mathbb{R}^S$ for which $sh_A(x)$ is the shadow of
$x$ on $A$ for every $x\in S$. The \emph{shadow polytope}
$P_{\text{shadow}}(G)$ of $G$ is defined as the polytope spanned by
the shadow vectors of all bases of $G$.
\end{definition}

For example, if $G=(E,\mathcal{F})$ is the undirected branching
greedoid of a graph $H$ with root node $r$, $A\in\mathcal{F}$ is the
edge set of a subtree $T=(V_T,A)$ of $H$ such that $r\in V_T$ and
$x\in E$ is arbitrary then it is easy to check that $sh_A(x)$ is the
number of nodes in $V_T$ that are unreachable via a path from $r$ in
$(V_T, A-x)$.  Obviously, in every greedoid $sh_A(x)=0$ if and only if
$x\notin A$. Furthermore, if $M$ is a matroid then ${sh_A(x)=1}$ is
obvious for every $x\in A$ and hence $sh_A$ is nothing but the
incidence vector of $A$. Consequently,
$P_{\text{shadow}}(M)=P_{\text{base}}(M)$ holds for every matroid $M$.

The significance of the notion of the shadow vector for local poset
greedoids is indicated by the following lemma: it shows that for every
weight function $c:S\rightarrow\mathbb{R}$ on the ground set, the
value of the objective function already seen in
Example~\ref{ex:dijkstra} is the dot product of the shadow vector and
$c$. This observation, together with Theorem~\ref{thm:boyd}, implies
that for local forest greedoids the greedy algorithm minimizes
non-negative, linear objective functions over the shadow
polytope. This fact will greatly be relied on in the proof of
Theorem~\ref{thm:main}.

Recall that $\Delta(X)$ denotes the unique base of a subfeasible set
$X\in\mathcal{F}^{\vee}$ in interval greedoids.

\begin{lemma}\label{lem:objfn}
Let $G=(S,\mathcal{F})$ be a local poset greedoid, 
$A\in\mathcal{F}$ and $c:S\rightarrow\mathbb{R}$ a weight
function. Then $\sum_{x\in A}c(x)\cdot sh_A(x)=\sum_{x\in A}c(P_x^A)$.
\end{lemma}
\begin{proof}
We claim that $y\in P_x^A$ if and only if $x\in A-\Delta(A-y)$ for
every $x,y\in A$. Indeed, $x\in\Delta(A-y)$ implies $P_x^A\subseteq
\Delta(A-y)$ by $\Delta(A-y)\in\mathcal{F}$, hence $x\in
A-\Delta(A-y)$ follows from $y\in P_x^A$. The converse follows from
the local union property: since $P_x^A,\Delta(A-y)\in\mathcal{F}$, 
$P_x^A,\Delta(A-y)\subseteq A$, $P_x^A\cup\Delta(A-y)\in\mathcal{F}$
holds and thus $y\in P_x^A$ by the definition of $\Delta(A-y)$.

Then the lemma follows by 
\[\sum_{x\in A}c(P^A_x)=\sum_{y\in A}|A-\Delta(A-y)|\cdot c(y)
=\sum_{y\in A}sh_A(y)\cdot c(y).\quad\qed\]
\end{proof}

We mentioned above that for matroids the shadow polytope and the
base polytope coincide. Therefore the following theorem, which is the
main result of this section, is indeed a direct generalization of
Theorem~\ref{thm:matroidbaseuphull}.

\begin{theorem}\label{thm:main}
Let $G=(S,\mathcal{F})$ be a local forest greedoid with rank function
$r$. Then
\[P_{\text{shadow}}^{\uparrow}(G)=\left\{x\in\mathbb{R}^S:
x(U)\ge r(S)-r(S-U)\mbox{ for all }U\subseteq
S\right\}.\]
\end{theorem}

To prepare the proof of Theorem~\ref{thm:main}, we need the following
lemmas. 

\begin{lemma}[\emph{Local Supermodularity Property}]\label{lem:lsp}
If $G=(S,\mathcal{F})$ is a local poset greedoid then $r(A)+r(B)\le
r(A\cup B)+r(A\cap B)$ holds for $A,B\subseteq S$ if $A\cup
B\in\mathcal{F}^{\vee}$. 
\end{lemma}
\begin{proof}
Let $X=\Delta(A)$ and $Y=\Delta(B)$. Then $X\cap Y\in\mathcal{F}$ by
the local intersection property. Furthermore, since for every feasible
set $Z\subseteq A\cap B$, $Z\cup X\in\mathcal{F}$ and $Z\cup
Y\in\mathcal{F}$ by the local union property, $Z\subseteq X\cap Y$
must hold by the definition of $\Delta(A)$ and $\Delta(B)$. Therefore
$X\cap Y=\Delta(A\cap B)$. Finally, since $X\cup Y\in\mathcal{F}$ is
also true by the local union property, we have
$r(A)+r(B)=|X|+|Y|=|X\cup Y|+|X\cap Y|\le r(A\cup B)+r(A\cap B)$ as
claimed.\qed
\end{proof}

Note that the above local supermodularity property also characterizes
local poset greedoids among all greedoids since it implies both
the local intersection and the local union properties if applied to
feasible sets. 

\begin{lemma}\label{lem:lpgsum}
If $G=(S,\mathcal{F})$ is a local poset greedoid,
$B\in\mathcal{F}^{\vee}$ is a subfeasible set and $\emptyset\ne
A\subseteq B$ then \[\sum_{x\in A}r(B-x)\le r(B-A)+(|A|-1)\cdot
r(B).\] 
\end{lemma}
\begin{proof}
We proceed by induction on $|A|$. The claim is trivial for $|A|=1$,
so let $|A|\ge 2$ and $A'=A-z$ for an arbitrary $z\in A$. Then 
\begin{multline*}
\sum_{x\in A}r(B-x)=\sum_{x\in A'}r(B-x)+r(B-z)\le
r(B-A')+(|A|-2)\cdot r(B)+r(B-z)\le\\
r(B-A)+r(B)+(|A|-2)\cdot r(B)=r(B-A)+(|A|-1)\cdot r(B),
\end{multline*}
where the first inequality follows by induction and the second by
Lemma~\ref{lem:lsp}.\qed
\end{proof}

\begin{myproposition}\label{prop:lpgsubset}
Let $G=(S,\mathcal{F})$ be a local poset greedoid,
$P_{\text{shadow}}(G)$ its shadow polytope and 
$Q=\left\{x\in\mathbb{R}^S:x(U)\ge r(S)-r(S-U)\mbox{ for all
}U\subseteq S\right\}$. Then 
$P_{\text{shadow}}^{\uparrow}(G)\subseteq Q$.
\end{myproposition}
\begin{proof}
Let $B$ be a base of $G$, $sh_B$ its shadow vector and $U\subseteq
S$. Then using Lemma~\ref{lem:lpgsum} we have
\begin{multline*}
sh_B(U)=\sum_{x\in U\cap B}\left(|B|-r(B-x)\right)=|U\cap B|\cdot r(S)-
\sum_{x\in U\cap B}r(B-x)\ge\\
|U\cap B|\cdot r(S)-r(B-U)-(|U\cap B|-1)\cdot r(S)=r(S)-r(B-U)\ge 
r(S)-r(S-U).
\end{multline*}
Therefore all vertices of $P_{\text{shadow}}(G)$ are in $Q$ which
implies $P_{\text{shadow}}(G)\subseteq Q$. Consequently,
$P_{\text{shadow}}^{\uparrow}(G)\subseteq Q^{\uparrow}=Q$.\qed
\end{proof}

It can happen that $P_{\text{shadow}}^{\uparrow}(G)$ is a proper
subset of $Q$ in the above proposition as shown by the example already
seen in Section~\ref{sect:prelimgreedy}: let $S=\{a,b,c,d\}$ and
$\mathcal{F}=\left\{\emptyset, \{a\}, \{b\}, \{a,b\}, \{a,d\},
\{a,b,c\}, \{a,c,d\}\right\}$. Then $G=(S,\mathcal{F})$ is a local
poset greedoid, the shadow vectors of its two bases are $(2,2,1,0)$
and $(3,0,1,2)$ (if the elements are arranged in alphabetical order),
both of which fulfill $2x_a+x_b\ge 6$, hence this inequality is
fulfilled by every member of
$P_{\text{shadow}}^{\uparrow}(G)$. However, $(2,1,1,1)\in Q$ is easy
to check which shows that
$Q-P_{\text{shadow}}^{\uparrow}(G)\ne\emptyset$. 

The claim of Theorem~\ref{thm:main} is that $\subseteq$ can be
replaced by $=$ in Proposition~\ref{prop:lpgsubset} in case of local
forest greedoids. The proof will follow the argument of Edmonds'
original proof of Theorem~\ref{thm:edmonds}: the greedy algorithm will
be used to construct an optimum dual solution. However, it should be
noted that the construction we give below is not an extension of that
of Edmonds: even if applied to matroids it gives a different optimum
dual solution. In particular, Edmonds' construction (even if adapted
to prove Theorem~\ref{thm:matroidbaseuphull}, which can easily be
done) yields a chain of subsets of the ground set which is not true
for the construction given below.

\begin{theorem}\label{thm:lfgdualsol}
Let $G=(S,\mathcal{F})$ be a local forest greedoid, $|S|=n$, 
$c:S\rightarrow\mathbb{R}^+$ a non-negative valued weight function,
$w(A)=\sum\{c(P_x^A):x\in A\}$ for every $A\in\mathcal{F}$ and $B_m$ a
minimum base with respect to $w$. Then there exist the subsets
$U_1,U_2,\ldots,U_n\subseteq S$ and corresponding values 
$y(U_1),y(U_2),\ldots,y(U_n)$ such that $y(U_i)\ge 0$ for all $1\le
i\le n$, $\sum\{y(U_i):x\in U_i\}=c(x)$ holds for every $x\in S$ and
$\sum_{i=1}^n\left(r(S)-r(S-U_i)\right)\cdot y(U_i)=w(B_m)$. 
\end{theorem}
\begin{proof}
Assume that a running of the greedy algorithm gives the base
$B=\{s_1,s_2,\ldots,s_r\}$ choosing the elements in this order and let
$B_1=\emptyset$ and $B_i=\{s_1,s_2,\ldots,s_{i-1}\}$ for every $2\le
i\le r$. Let $S-B=\{s_{r+1},\ldots,s_n\}$ with the elements ordered
arbitrarily. Finally, denote $P_0=\emptyset$ and $P_i=P_{s_i}^{B}$ for
every $1\le i\le r$.  Then let
\[\begin{array}{r@{}c@{}l@{\quad}l}
U_i&=&\left\{\begin{array}{ll}
\Gamma(B_i),&\text{if }1\le i\le r,\\
\{s_i\},&\text{if }r<i\le n\\
\end{array}\right.&\text{and}\\[3\medskipamount]
y(U_i)&=&\left\{\begin{array}{ll}
c(P_i)-c(P_{i-1}),&\text{if }1\le i\le r,\\
c(s_i)-\sum\{y(U_j):j\le r, s_i\in U_j\},&\text{if }r<i\le n.\\
\end{array}\right.\end{array}\]

We prove that the above choice of $U_i$ and $y(U_i)$ fulfills all
requirements of the theorem through a series of claims. 

\newcommand{\claimqed}{\hfill$\Diamond$}

\begin{myclaim}\label{cl:1}
Let $x\in\bigcup_{i=1}^rU_i$ and $j=\min\{i:x\in U_i\}$.
Then $P_x^{B_j+x}=P_{j-1}+x$.
\end{myclaim}
\begin{proof}
If $j=1$ then $P_x^{B_j+x}=\{x\}$ and thus the claim is obvious, so
assume $j\ge2$ and hence $|P_x^{B_j+x}|\ge2$.  Since $x\in
U_j-U_{j-1}$, we have $B_j+x\in\mathcal{F}$ and
$B_{j-1}+x\notin\mathcal{F}$.  Since
$B_{j-1},P_x^{B_j+x},B_j+x\in\mathcal{F}$, $B_{j-1}\cup
P_x^{B_j+x}\in\mathcal{F}$ follows from the local union property. This
implies $s_{j-1}\in P_x^{B_j+x}$ by $B_{j-1}+x\notin\mathcal{F}$.

Since $P_x^{B_j+x}-x\in\mathcal{F}$ and $s_{j-1}\in
P_x^{B_j+x}-x\subseteq B_j$, $P_{j-1}\subseteq P_x^{B_j+x}-x$ and
therefore $P_{j-1}+x\subseteq P_x^{B_j+x}$ follows from the definition
of a path. The second to last element in the unique ordering of
$P_x^{B_j+x}$ is obviously $s_{j-1}$ otherwise $s_{j-1}\in
P_t\subseteq B_{t+1}$ would follow from Theorem~\ref{thm:lfgpath} for
some $t<j-1$, a contradiction. Therefore $P_x^{B_j+x}=P_{j-1}+x$ as
claimed.\claimqed
\end{proof}

\begin{myclaim}\label{cl:2}
Let $x\in\bigcup_{i=1}^rU_i$, $j=\min\{i:x\in U_i\}$, $k=\max\{i:x\in
U_i,i\le r\}$. Then $x\in U_i$ holds for every $j\le i\le k$ and
$\sum\{y(U_i):x\in U_i,i\le r\}=c(P_k)-c(P_{j-1})$.
\end{myclaim}
\begin{proof}
From $x\in U_j\cap U_k$ we have $B_j+x\in\mathcal{F}$ and
$B_k+x\in\mathcal{F}$ which, by the local union property, imply
$B_i+x\in\mathcal{F}$ and therefore $x\in U_i$ for every $j\le i\le
k$ as claimed. Consequently, 
\[\sum\{y(U_i):x\in U_i,i\le
r\}=\sum_{i=j}^k\left(c(P_i)-c(P_{i-1})\right)=c(P_k)-c(P_{j-1}).
\quad\Diamond\]
\end{proof}

\begin{myclaim}\label{cl:3}
$c(P_1)\le c(P_2)\le \ldots\le c(P_r)$.
\end{myclaim}
\begin{proof}
If $s_i\in U_{i-1}$ for some $2\le i\le r$ then $c(P_{i-1})\le c(P_i)$
is implied by the fact that the greedy algorithm could have chosen
$s_i$ instead of $s_{i-1}$. If, on the other hand, $s_i\notin U_{i-1}$
then $P_i=P_{i-1}+s_i$ follows from Claims~\ref{cl:1} and
\ref{cl:2}. Hence $c(P_i)=c(P_{i-1})+c(s_i)$ which proves the
claim.\claimqed 
\end{proof}

\begin{myclaim}
$y(U_i)\ge 0$ for all $1\le i\le n$
and $\sum\{y(U_i):x\in U_i\}=c(x)$ for all $x\in S$. 
\end{myclaim}
\begin{proof}
Let first $x=s_t$ for some $1\le t\le r$. Then
$y(U_t)\ge0$ is immediate from Claim~\ref{cl:3}. Define $j$ and $k$ as
in Claim~\ref{cl:3}. Then $k=t$ is obvious and Claim~\ref{cl:1} gives 
$P_t=P_{j-1}+s_t$. Therefore from Claim~\ref{cl:2} we have 
$\sum\{y(U_i):s_i\in U_i\}=c(P_t)-c(P_{j-1})=c(s_t)$ as claimed.

Now let $x=s_t$ for some $r<t\le n$. If $s_t\notin\bigcup_{i=1}^rU_i$
then $\sum\{y(U_i):s_t\in U_i\}=y(U_t)=c(s_t)\ge0$ is clear. 
If, on the other hand, $s_t\in\bigcup_{i=1}^rU_i$ then again define
$j$ and $k$ as in Claim~\ref{cl:3}. Since the greedy algorithm could
have chosen $x$ instead of $s_k$ by $x=s_t\in U_k$, 
we have $c(P_x^{B_k+x})\ge c(P_k)$. Furthermore, $B_j+x\subseteq
B_k+x$ implies $P_x^{B_j+x}=P_x^{B_k+x}$. These, together with 
Claims~\ref{cl:1} and \ref{cl:2} imply 
\[c(x)=c(P_x^{B_k+x})-c(P_{j-1})\ge c(P_k)-c(P_{j-1})=
\sum\{y(U_i):x\in U_i,i\le r\},\]
hence we have the claim by the definitions of $U_i$ and
$y(U_i)$.\claimqed  
\end{proof}

\begin{myclaim}\label{cl:5}
\[r(S-U_i)=\left\{\begin{array}{ll}
i-1,&\text{if }1\le i\le r,\\
r(S),&\text{if }r<i\le n\\
\end{array}\right.\]
\end{myclaim}
\begin{proof}
If $r<i$ then $B\subseteq S-U_i$ so $r(S-U_i)=r(S)$ is
obvious. For $1\le i\le r$ we show that $B_i$ is a base of $S-U_i$
which will settle the claim by $|B_i|=i-1$. 
$B_i\subseteq S-U_i$ and $B_i\in\mathcal{F}$ are obvious. Furthermore,
if $|B_i|<|X|$ and $X\in\mathcal{F}$ then $B_i+x\in\mathcal{F}$ for
some $x\in X-B_i$ by \ref{prop:exchange} and hence $x\in U_i$, which
proves that $B_i$ is indeed a base of $S-U_i$.\claimqed 
\end{proof}

Finally, it remains to show that 
$\sum_{i=1}^n\left(r(S)-r(S-U_i)\right)\cdot y(U_i)=w(B_m)$
holds. Using Claim~\ref{cl:5} we get 
\begin{multline*}
\sum_{i=1}^n\left(r(S)-r(S-U_i)\right)\cdot y(U_i)=
\sum_{i=1}^r\left(r(S)-i+1\right)\cdot y(U_i)=\\
\sum_{i=1}^r\left(r-i+1\right)\cdot\left(c(P_i)-c(P_{i-1})\right)=
\sum_{i=1}^r c(P_i)=w(B)=w(B_m),
\end{multline*}
where the last equation follows from Theorem~\ref{thm:boyd}.\qed 
\end{proof}

Now we are ready for the

\begin{theopargself}
\begin{proof}[of Theorem~\ref{thm:main}.]
Let $P=P_{\text{shadow}}^{\uparrow}(G)$ for short. 
By Proposition~\ref{prop:lpgsubset} we have
$P\subseteq Q$, where
$Q=\left\{x\in\mathbb{R}^S:x(U)\ge r(S)-r(S-U)\mbox{ for all
}U\subseteq S\right\}$. To show equality it suffices to prove that
$\min\{cx:x\in P\}=\min\{cx:x\in Q\}$
holds for every $c\in\mathbb{R}^S$, $c\ge0$.
(Indeed, since $P^{\uparrow}=P$ holds, $P$ can be written in the form
$P=\{x:Ax\ge b\}$ for some matrix $A\ge0$. If a $z\in Q-P$ existed
then $z$ would violate a constraint $cx\ge\delta$ of $Ax\ge b$ and
hence $\min\{cx:x\in P\}>\min\{cx:x\in Q\}$ would follow.)

So let a $c\in\mathbb{R}^S$, $c\ge0$ be fixed, let
$w(A)=\sum\{c(P_x^A):x\in A\}$ for every $A\in\mathcal{F}$ and $B_m$ a
minimum base with respect to $w$. Using Lemma~\ref{lem:objfn} and
since $\min\{cx:x\in P_{\text{shadow}}(G)\}$ is attained on a vertex
of $P_{\text{shadow}}(G)$ and $P_{\text{shadow}}(G)\subseteq
P\subseteq Q$, we get
\begin{multline}\label{eq:main}
w(B_m)=\min\left\{\sum_{s\in B}c(P_s^B):B\in\mathcal{B}\right\}=
\min\left\{\sum_{s\in B}c(s)\cdot sh_B(s):B\in\mathcal{B}\right\}=\\
\min\{cx:x\in P_{\text{shadow}}(G)\}\ge\min\{cx:x\in
P\}\ge\min\{cx:x\in Q\}. 
\end{multline}

From the duality theorem of linear programming we get 
\begin{multline*}\min\{cx:x\in Q\}=
\max\Big\{\sum y(U)(r(S)-r(S-U)):\\
\sum\{y(U):s\in U\}=c(s)
\text{ for all }s\in S,y(U)\ge0\text{ for all }U\subseteq S\Big\}.
\end{multline*}

Theorem~\ref{thm:lfgdualsol} implies that this maximum is at least
$w(B_m)$, which in turn implies that every inequality in
(\ref{eq:main}) is fulfilled with equation and hence concludes the
proof.\qed
\end{proof}
\end{theopargself}

\begin{mycorollary}\label{cor:intsol}
If $G=(S,\mathcal{F})$ is a local forest greedoid and
$c:S\rightarrow\mathbb{Z}^+$ is a non-negative integer valued weight
function then the linear programming problem 
\[\min\big\{cx:x(U)\ge r(S)-r(S-U)\mbox{ for all
}U\subseteq S\big\}\]
and its dual
\begin{multline*}
\max\Big\{\sum y(U)(r(S)-r(S-U)):\sum\{y(U):s\in U\}=c(s)
\text{ for all }s\in S,\\
y(U)\ge0\text{ for all }U\subseteq S\Big\}
\end{multline*}
have integer optimum solutions.
\end{mycorollary}
\begin{proof}
It follows from the proof Theorem~\ref{thm:main} that the minimum of
the primal problem is attained on the shadow vector of a base of $G$
which is obviously integer. Furthermore, the construction of the proof
of Theorem~\ref{thm:lfgdualsol} yields an integer optimum solution of
the dual problem if $c$ is integer.\qed
\end{proof}

\begin{mycorollary}
If $G=(S,\mathcal{F})$ is a local forest greedoid then the system 
\[x(U)\ge r(S)-r(S-U)\mbox{ for all
}U\subseteq S\]
is totally dual integral.
\end{mycorollary}
\begin{proof}
Immediately from Corollary~\ref{cor:intsol} after observing that the
minimum of the primal program clearly does not exist if $c$ contains
a negative component.\qed
\end{proof}

We remark that no similar description of $P_{\text{shadow}}(G)$ is to
be hoped for, not even for branching greedoids. Indeed, it follows
from Lemma~\ref{lem:objfn} that maximizing a linear objective function
over $P_{\text{shadow}}(G)$ translates to maximizing $\sum_{e\in
  E(H)}c(P_e)$ which is, as it was pointed out in \cite[Chapter
  XI.]{KLS}, NP-hard as it contains the Hamilton path problem.
Therefore the existence of such a description of
$P_{\text{shadow}}(G)$ would imply that, for example, the Hamilton
path problem is in co-NP, which is highly unlikely.

\section{An Application: Reliability of Networks via Game
  Theory}\label{sect:apply} 

The problem of measuring the robustness or reliability of a graph
arises in many applications. The most widely applied reliability
metrics are obviously the connectivity based ones, however, these are
unsuitable in many cases -- for example because in many
applications the network is almost completely functional if removing
some nodes or links results in the loss of only a small number of
nodes that are in some sense insignificant or peripheral.  

Applying game-theoretical tools for measuring the reliability of a
graph has become very common. The basic idea is very natural: define a
game between two virtual players, the Attacker and the Defender, such
that the rules of the game capture the circumstances under which
reliability is to be measured. Then analyzing the game might give rise
to an appropriate security metric: the better the Attacker can do in
the game, the lower the level of reliability is. This kind of analysis
can give rise to new graph reliability metrics and in some cases it
can shed a new light on some well-known ones. 

To illustrate this, consider the following \emph{Spanning Tree Game}:
a connected, undirected graph $G$, a positive valued damage function
$d:E(G)\rightarrow\mathbb{R}^+$ and a cost function
$c:E(G)\rightarrow\mathbb{R}$ are given. For each edge, $d(e)$
represents the ``damage'' caused by the loss of $e$ (or in other
words, the ``importance'' of $e$) and $c(e)$ represents the cost of
attacking $e$. The Attacker chooses (or ``attacks and destroys'') an
edge $e$ of $G$ and the Defender (without knowing the Attacker's
choice) chooses a spanning tree $T$ of $G$ (that she intends to use as
some kind of ``communication infrastructure'').  Regardless of the
Defender's choice, the Attacker has to pay the cost of attack $c(e)$
to the Defender. There is no further payoff if $e\notin T$. If, on the
other hand, $e\in T$ then the Defender pays the Attacker the damage
value $d(e)$. Since this game is a two-player, zero-sum game, it has a
unique Nash-equilibrium payoff (or, in simpler terms, game value) $V$
by Neumann's classic Minimax Theorem. Since $V$ is the highest
expected gain the Attacker can guarantee himself by an appropriately
chosen mixed strategy (that is, probability distribution on the set of
edges), it makes sense to say that $\frac1V$ is a valid reliability
metric.

After some preliminary results on some special cases in the literature
(see \cite{Sz2} for the details), the Spanning Tree Game was solved in
the above defined general form in \cite{Sz}. In fact, it was
considered there in a more general, matroidal setting: the
\emph{Matroid Base Game} was defined analogously to the Spanning Tree
Game with the only difference being that the Attacker chooses an
element of the ground set of a matroid $M=(S,\mathcal{F})$ and the
Defender chooses a base $B$ of $M$. Then the following result was
proved.

\begin{theorem}[\cite{Sz}]\label{thm:mbgame}
For every input of the Matroid Base Game the game value is 
\[\max_{\emptyset\ne U\subseteq S}
\frac{r(S)-r(S-U)-q(U)}{p(U)},\]
where $p(s)=\frac1{d(s)}$ and $q(s)=\frac{c(s)}{d(s)}$ for all 
$s\in S$. Furthermore, if $M$ is given by an independence testing
oracle then there exists a strongly polynomial algorithm that computes  
the game value of the Matroid Base Game and an optimum mixed
strategy for both players. 
\end{theorem}

The running time of the algorithm given in \cite{Sz} was later
substantially improved in \cite{BB}.

If specialized to the Spanning Tree Game and to the $c\equiv0$ case,
the above theorem implies that the game value is the reciprocal of a
well-known graph reliability metric: the \emph{strength} of a graph is
defined as \[\sigma_p(G)=\min\left\{\frac{p(U)}{\mathop{\rm
comp}(G-U)-1}: U\subseteq E(G),\mathop{\rm comp}(G-U)>1\right\},\]
where ${\mathop{\rm comp}(G-U)}$ is the number of components of the
graph obtained from $G$ by deleting $U$ and
$p:E(G)\rightarrow\mathbb{R}^+$ is a weight function. This notion was
defined in the weighted case and its computability in strongly
polynomial time was proved in \cite{Cun}.

While the Matroid Base Game has further relevant applications beyond
the Spanning Tree Game (see \cite{Sz}), there are other types of games
of a similar nature which do not fit into this framework. The
following \emph{Rooted Spanning Tree Game} was considered in
\cite{BLS2}: a (mixed) graph $H$ with a ``headquarters'' node $r$ is
given such that every node is reachable from $r$.  (The role of $r$
can be that all other nodes need to communicate with $r$ only, for
example to transmit some collected data to $r$.)  Furthermore assume
that a cost function $c:E(H)\rightarrow\mathbb{R}$ is also
given. Again, the Attacker chooses an edge $e$, the Defender
chooses a spanning tree $T$ and the cost of attack $c(e)$ is
payed by the Attacker to the Defender in all cases and there is no
further payoff if $e\notin T$. However, if $e\in T$ then the payoff
from the Defender to the Attacker is the number of nodes that become
unreachable from $r$ in $T$ after removing $e$.

Since this number is nothing but the shadow $sh_T(e)$ in case of the
branching greedoid, the definition of the \emph{Local Forest Greedoid
  Base Game} presents itself: given a local forest greedoid
$G=(S,\mathcal{F})$ and weight functions $d,c\in\mathbb{R}^S$ with
$d>0$, the Attacker chooses an element $s\in S$, the Defender chooses
a base $B$ of $G$ and then the payoff from the Defender to the
Attacker is $d(s)\cdot sh_B(s)-c(s)$. Clearly, this game is a direct
generalization of the Matroid Base Game mentioned above. Then, using
Theorem~\ref{thm:main} and following the proof of \cite[Theorem 5]{Sz}
we can prove the following.

\begin{theorem}\label{thm:lfggame}
For every input of the Local Forest Greedoid Base Game the game value
is 
\[\max_{\emptyset\ne U\subseteq S}
\frac{r(S)-r(S-U)-q(U)}{p(U)},\]
where $p(s)=\frac1{d(s)}$ and $q(s)=\frac{c(s)}{d(s)}$ for all 
$s\in S$.
\end{theorem}
\begin{proof}
Denote the game value by $V$ and 
assume that a mixed strategy of the Defender
$\{\delta(B):B\in\mathcal{B}\}$ (that is, a probability distribution
$\delta$ on $\mathcal{B}$) is given. Then assuming that the Attacker
chooses a given fixed 
element $s\in S$ in the game, the Defender's expected loss is 
\begin{equation}\label{eq:defpay}
\sum_{B\in\mathcal{B}}\delta(B)\cdot\big(d(s)\cdot sh_B(s)-c(s)\big)=
d(s)\cdot\left(\sum_{B\in\mathcal{B}}\delta(B)\cdot sh_B(s)\right)-c(s).
\end{equation}
Let $x(s)=\sum\{\delta(B)\cdot sh_B(s):B\in\mathcal{B}\}$ for all
$s\in S$. Then the vector $x\in\mathbb{R}^S$ is nothing but an element
of $P_{\text{shadow}}(G)$ by definition (since the values
$\delta(B)$ form the set of coefficients of a convex
combination). Since, by definition, the Defender's objective is to
minimize the maximum expected loss she has to suffer, her task amounts
to the following by (\ref{eq:defpay}):
\begin{equation}\label{eq:defjob1}
\min\big\{\mu: \exists x\in P_{\text{shadow}}(G), 
d(s)\cdot x(s)-c(s)\le\mu\mbox{ for all }s\in S\big\}.
\end{equation}
In other words, the minimum in (\ref{eq:defjob1}) is equal to $V$
by Neumann's Minimax Theorem. Rearranging (\ref{eq:defjob1}):
\begin{equation*}
V=\min\big\{\mu: \exists x\in P_{\text{shadow}}(G), 
x\le\mu\cdot p+q\}.
\end{equation*}
Using the definition of $P_{\text{shadow}}^{\uparrow}(G)$
this is further equivalent to the following: 
\begin{equation}\label{eq:defjob3}
V=\min\big\{\mu: \mu\cdot p+q\in P_{\text{shadow}}^{\uparrow}(G)\}.
\end{equation}
By Theorem~\ref{thm:main}
$\mu\cdot p+q\in P_{\text{shadow}}^{\uparrow}(G)$ is true if
and only if 
\begin{equation*}
\mu\cdot p(U)+q(U)\ge r(S)-r(S-U)
\end{equation*}
holds for all $U\subseteq S$. Then simple rearranging (and observing
that this inequality is trivial for $U=\emptyset$) immediately gives
that $\mu\cdot p+q\in P_{\text{shadow}}^{\uparrow}(G)$ is true if and
only if
\begin{equation*}
\mu\ge\max_{\emptyset\ne U\subseteq S}
\frac{r(S)-r(S-U)-q(U)}{p(U)}.
\end{equation*}
Hence $V$, the minimum of all such $\mu$'s is exactly this maximum.\qed
\end{proof}

If specialized to the branching greedoid and to the $c\equiv0$ case it
follows that the value of the Rooted Spanning Tree Game is the
reciprocal of another known graph reliability metric, also defined in
\cite{Cun}. Interested readers are referred to \cite{Sz2} for the
details. Furthermore, the above theorem also generalizes the first
statement of Theorem~\ref{thm:mbgame}. However, generalizing the
algorithmic statement of Theorem~\ref{thm:mbgame} to the Local Forest
Greedoid Base Game is left as an open problem.

\section{Optimality of the Greedy Algorithm in
  Greedoids}\label{sect:greedyalg}

We start with the following theorem which seems to be new, but its
proof is just an adaptation of that of the result of Korte and
Lov\'asz \cite{KL2}, \cite[Theorem~XI.2.2]{KLS} mentioned at the end
of Section~\ref{sect:prelimgreedy} on the optimality of the greedy
algorithm in case of linear objective functions.

\begin{theorem}\label{thm:greedyopt}
Let $G=(S,\mathcal{F})$ be an arbitrary greedoid and
$w:\mathcal{F}\rightarrow\mathbb{R}$ an objective function that
fulfills the following property:
\begin{numberedlist}[start=1]
\item \label{prop:greedyopt}
If for some $A\subseteq B$, $A\in\mathcal{F}$, $A+x\in\mathcal{F}$,
$B\in\mathcal{B}$ and $x\in S-B$ it holds that $w(A+x)\ge w(A+u)$ for
every $u\in\Gamma(A)$ then there exists a $y\in B-A$ such that
$B-y+x\in\mathcal{B}$ and $w(B-y+x)\ge w(B)$.
\end{numberedlist}
Then the greedy algorithm gives a maximum base with respect to $w$. 
\end{theorem}
\begin{proof}
Assume by way of contradiction that the greedy algorithm gives 
the base $B_g=\{a_1,a_2,\ldots,a_r\}$ choosing the elements in
this order, but $B_g$ is not maximum with respect to $w$. Choose a
maximum base $B_m$ with respect to $w$ such that
$\max\{i:a_1,\ldots,a_i\in B_m\}$ is maximum possible, let this
maximum be $k$ and $A=\{a_1,\ldots,a_k\}$. Then $A\in\mathcal{F}$,
$A\subseteq B_m$, $a_{k+1}\notin B_m$ and $w(A+a_{k+1})\ge w(A+u)$ for
every $u\in\Gamma(A)$ by the operation of the greedy
algorithm. Therefore, by \ref{prop:greedyopt}, there exists a $y\in
B_m-A$ such that $B_m-y+a_{k+1}\in\mathcal{B}$ and
$w(B_m-y+a_{k+1})\ge w(B_m)$. Therefore $B_m-y+a_{k+1}$ is also a
maximum base with respect to $w$, but
$\{a_1,\ldots,a_k,a_{k+1}\}\subseteq B_m-y+a_{k+1}$ contradicts the
choice of $B_m$.\qed
\end{proof}

It is worth noting that, in spite of its simplicity, the above theorem
implies the optimality of the greedy algorithm in all three examples
listed in Section~\ref{sect:prelimgreedy}. This is easy to check in
case of Examples~\ref{ex:matroid} and \ref{ex:prim} and in case of
Example~\ref{ex:dijkstra} it will follow from the results
below. Furthermore, it is not too hard to show that
Theorem~\ref{thm:greedyopt} also implies Theorem~\ref{thm:kl} in case
of objective functions $w:\mathcal{F}\rightarrow\mathbb{R}$ that are
independent of the ordering. (This could be proved by an argument
similar to that of Theorem~\ref{thm:intgreedyopt} below, we omit the
details here.)

Moreover, Theorem~\ref{thm:greedyopt} is in a sense best possible as
shown by the following theorem. To claim the theorem, we need to
extend the definition of minors of greedoids given in
Section~\ref{sect:greedoid} to incorporate modifying the objective
function $w_G:\mathcal{F}\rightarrow\mathbb{R}$ in an obvious way: in
case of a deletion $G\setminus X$ $w_G$ is simply restricted to $S-X$,
while in case of a contraction $G/X$ the modified objective function
becomes $w_{G/X}(A)=w_G(A\cup X)$.

\begin{theorem}
Assume that the objective function
$w_G:\mathcal{F}\rightarrow\mathbb{R}$ violates condition
\ref{prop:greedyopt} for a greedoid $G=(S,\mathcal{F})$. Then there
exists a minor $H$ of $G$ such that a legal running of the greedy
algorithm on $H$ gives a base that is not maximum with respect to
$w_H$. 
\end{theorem}
\begin{proof}
Assume that \ref{prop:greedyopt} is violated by an $A\in\mathcal{F}$,
$B\in\mathcal{B}$ and $x\in S-B$. Let $Y=S-B-x$ and $H=(G\setminus
Y)/A$. Then the greedy algorithm run on $H$ with respect to $w_H$ can
start with $x$ since $w_H(\{x\})\ge w_H(\{u\})$ for every $u\in
\Gamma(\emptyset)$ holds in $H$ by \ref{prop:greedyopt}. Therefore
this running of the greedy algorithm terminates with a base $B_g$ of
$H$ such that $x\in B_g$. Since the ground set of $H$ is $S_H=B-A+x$
and $B-A$ is a base of $H$, $S_H-B_g=\{y\}$ for some $y\in B-A$. Since
\ref{prop:greedyopt} is violated by $A$, $B$ and $x$, we have
$w(B-y+x)<w(B)$. Consequently, $w_H(B_g)=w(B-y+x)<w(B)=w_H(B-A)$ which
proves that $B_g$ is not maximum with respect to $w_H$.\qed
\end{proof}

The following theorem will be weaker than Theorem~\ref{thm:greedyopt}
-- not only because it applies to interval greedoids only, but also
because it will not cover Example~\ref{ex:prim} given in
Section~\ref{sect:prelimgreedy} (or the case of linear objective
functions in general). However, it can also be regarded as a corrected
version of the faulty condition \ref{prop:klrossz} mentioned in
Section~\ref{sect:prelimgreedy} and it will be easier to work with
later on.

\begin{theorem}\label{thm:intgreedyopt}
Let $G=(S,\mathcal{F})$ be an interval greedoid and
$w:\mathcal{F}\rightarrow\mathbb{R}$ an objective function that
fulfills the following property:
\begin{numberedlist}[start=2]
\item \label{prop:intgreedyopt} If for some $A\subseteq B$,
  $A\in\mathcal{F}$, $A+x\in\mathcal{F}$, $x,z\in S-B$ and 
  $B+x\in\mathcal{B}$, $B+z\in\mathcal{B}$ such
  that 
  $\Delta(B)\cup\{x,z\}\notin\mathcal{F}$ it holds that $w(A+x)\ge
  w(A+u)$ for every $u\in\Gamma(A)$ then $w(B+x)\ge w(B+z)$.
\end{numberedlist}
Then the greedy algorithm gives a maximum base with respect to $w$. 
\end{theorem}
\begin{proof}
We will show that \ref{prop:intgreedyopt} implies \ref{prop:greedyopt}
which will obviously settle the proof by Theorem~\ref{thm:greedyopt}. 
So let $A$, $B$ and $x$ be given such that $A\subseteq B$,
$A,A+x\in\mathcal{F}$, $B\in\mathcal{B}$ and $w(A+x)\ge w(A+u)$ for
every $u\in\Gamma(A)$. We need to show the existence of a $y\in B-A$
according to \ref{prop:greedyopt}. 

Let $(b_1,\ldots,b_k)$ be a feasible ordering of $A$ and, using
\ref{prop:exchange}, augment this repeatedly to get a feasible
ordering $(b_1,\ldots,b_k,b_{k+1},\ldots,b_r)$ of $B$. Denote
$B_0=\emptyset$ and $B_i=\{b_1,\ldots,b_i\}$ for every $1\le i\le
r$. Let $t\in\{1,\ldots,r\}$ be the largest index such that
$B_{t-1}+x\in\mathcal{F}$. Obviously, $t$ exists and $t\ge k+1$ since
$B_k+x=A+x\in\mathcal{F}$. Now set $y=b_t$; we claim that this is a
suitable choice for \ref{prop:greedyopt}.

Trivially, $y\in B-A$ by $t\ge k+1$. To show 
$B-y+x\in\mathcal{F}$, augment $B_{t-1}+x$ from $B_{t+1}$; then 
augment the obtained feasible set from $B_{t+2}$ and continue like
this until a base is obtained. Then $b_t$ can never occur as an 
augmenting element during this process by the choice of $t$ which 
implies $B-y+x\in\mathcal{F}$ as claimed. 

Let $C=B-y$. We claim that $\Delta(C)\cup\{x,y\}\notin\mathcal{F}$, so
assume the opposite towards a contradiction.  Since
$B_{t-1}\in\mathcal{F}$ and $B_{t-1}\subseteq C$, we have
$B_{t-1}\subseteq\Delta(C)$. Furthermore, $B_{t-1}+x\in\mathcal{F}$ by
the choice of $t$ and $B_{t-1}+y=B_t\in\mathcal{F}$ is also
true. Since $B_{t-1}+x,B_{t-1}+y\subseteq\Delta(C)\cup\{x,y\}$,
$B_{t-1}\cup\{x,y\}=B_t+x\in\mathcal{F}$ follows by the local union
property \ref{lup}. This either contradicts the choice of $t$ if $t<r$
or the fact that $B$ is a base if $t=r$.

Consequently, since we have $C+y=B\in\mathcal{B}$,
$C+x=B-y+x\in\mathcal{B}$ and $w(A+x)\ge w(A+u)$ for every
$u\in\Gamma(A)$, we get $w(C+x)\ge w(C+y)$ from
\ref{prop:intgreedyopt}, which concludes the proof by $C+x=B-y+x$ and
$C+y=B$.\qed
\end{proof}

The next theorem gives a generalization of Theorem~\ref{thm:boyd}. 

\begin{theorem}\label{thm:dijkstra}
Let $G=(S,\mathcal{F})$ be a local forest greedoid, $\mathcal{P}$
its set of paths and $f:\mathcal{P}\rightarrow\mathbb{R}$ a
function that satisfies the following monotonicity constraints:
\begin{enumerate}[label=(\roman*)]
\item if $A,B\in\mathcal{P}$ and $A\subseteq B$ then $f(A)\le f(B)$;
\item if $A,B,A\cup C,B\cup C\in\mathcal{P}$ and $f(A)\le f(B)$ then 
$f(A\cup C)\le f(B\cup C)$.
\end{enumerate}
Finally, let $w(A)=\sum\{f(P_x):x\in A\}$ for every
  $A\in\mathcal{F}$. Then the greedy algorithm gives a minimum base
  with respect to $w$.
\end{theorem}

We will need the following lemma for proving the above theorem.

\begin{lemma}\label{lem:dijkstra}
Let $G=(S,\mathcal{F})$ be a local poset greedoid, $B\subseteq S$
and $x,z\in S-B$ such that $B+x\in\mathcal{F}$, $B+z\in\mathcal{F}$
and $\Delta(B)\cup\{x,z\}\notin\mathcal{F}$. Then
$P_e^{B+x}\cap(B-\Delta(B))=P_e^{B+z}\cap(B-\Delta(B))$ holds for every
$e\in B-\Delta(B)$.
\end{lemma}
\begin{proof}
Since no feasible set in $B$ can contain $e$ by $e\in B-\Delta(B)$ and
the local union property \ref{lup}, we have $x\in P_e^{B+x}$ and $z\in
P_e^{B+z}$. Let $D_i=P_e^{B+i}\cap\Delta(B)$ and
$H_i=P_e^{B+i}\cap(B-\Delta(B))$ for $i\in\{x,z\}$. We need to show
$H_x=H_z$.

Since $P_e^{B+x},\Delta(B),B+x\in\mathcal{F}$ and 
$P_e^{B+x},\Delta(B)\subseteq B+x$, the local union property
implies $\Delta(B)\cup H_x+x\in\mathcal{F}$.

We claim that $\Delta(B)\cup H_x+z\in\mathcal{F}$. To show this, first
observe that augmenting $\Delta(B)$ from $B+x$ and $B+z$ implies 
$\Delta(B)+x,\Delta(B)+z\in\mathcal{F}$ by the definition of
$\Delta(B)$. Therefore $\Delta(B)\cup\{x,z\}\notin\mathcal{F}$ also
implies $\Delta(B)\cup\{x,z\}\notin\mathcal{F}^{\vee}$ by the local
union property. Consequently, repeatedly augmenting $\Delta(B)+z$ from 
$\Delta(B)\cup H_x+x$ yields $\Delta(B)\cup H_x+z\in\mathcal{F}$ as
claimed since $x$ can not augment. 

Then since $P_e^{B+z},\Delta(B)\cup H_x+z,B+z\in\mathcal{F}$ and
$P_e^{B+z},\Delta(B)\cup H_x+z\subseteq B+z$, the local intersection
property \ref{lip} implies $D_z\cup (H_x\cap H_z)+z\in\mathcal{F}$.
Since $P_e^{B+z}=D_z\cup H_z+z$, $H_z\subseteq H_x$ must hold by the
definition of a path. By symmetry we also have $H_x\subseteq H_z$,
which completes the proof.\qed
\end{proof}

Now we are ready for proving Theorem~\ref{thm:dijkstra}. The proof
follows the argument of \cite[page 156]{KLS} where they showed that
property~\ref{prop:klrossz} is fulfilled by a similarly defined
objective function $w$ in local poset greedoids.  As mentioned in
Section~\ref{sect:prelimgreedy}, that was insufficient for
guaranteeing the optimality of the greedy algorithm, however, a
similar argument will work well with Theorem~\ref{thm:intgreedyopt}.

\begin{theopargself}
\begin{proof}[of Theorem~\ref{thm:dijkstra}.]
We will show that \ref{prop:intgreedyopt} is fulfilled by $(-w)$. So
let $A$, $B$, $x$ and $z$ given such that $A,A+x\in\mathcal{F}$,
$x,z\in S-B$, $B+x,B+z\in\mathcal{B}$, 
$\Delta(B)\cup\{x,z\}\notin\mathcal{F}$ and $w(A+x)\le w(A+u)$ for
every $u\in\Gamma(A)$ hold. We need to show $w(B+x)\le w(B+z)$.

Since $\Delta(B)\in\mathcal{F}$, we have 
\setcounter{equation}{5}
\begin{equation}\label{eq:dijkstra}
w(B+i)=w(\Delta(B))+f(P_i^{B+i})+\sum_{e\in B-\Delta(B)}f(P_e^{B+i})
\end{equation}
for $i\in\{x,z\}$.  Let $(b_1,\ldots,b_k=z)$ be the unique feasible
ordering of $P_z^{B+z}$ according to Theorem~\ref{thm:lfgpath} and let
$j\in\{1,\ldots,k\}$ be the smallest index such that $b_j\notin A$ and
denote $u=b_j$. Then since
$\{b_1,\ldots,b_j\},A,P_z^{B+z}\in\mathcal{F}$ and
$\{b_1,\ldots,b_j\},A\subseteq P_z^{B+z}$, we have $A+u\in\mathcal{F}$
by the local union property. Therefore $w(A+x)\le w(A+u)$, which
implies $f(P_x^{A+x})\le f(P_u^{A+u})$ by $w(A+i)=w(A)+f(P_i^{A+i})$
for $i\in\{x,u\}$. Furthermore, $P_u^{B+z}=\{b_1,\ldots,b_j\}$ by
Theorem~\ref{thm:lfgpath}, which implies $f(P_u^{B+z})\le
f(P_z^{B+z})$ by property (i). Noting that $P_x^{A+x}=P_x^{B+x}$ and
$P_u^{A+u}=P_u^{B+z}$ are obvious by $A+x,A+u\in\mathcal{F}$, these
together imply $f(P_x^{B+x})\le f(P_z^{B+z})$.

First assume $B\in\mathcal{F}$. Then $B=\Delta(B)$ and hence
$w(B+i)=w(\Delta(B))+f(P_i^{B+i})$ follows from (\ref{eq:dijkstra})
for $i\in\{x,z\}$. Therefore $w(B+x)\le w(B+z)$ follows immediately
from $f(P_x^{B+x})\le f(P_z^{B+z})$.

Now assume $B\notin\mathcal{F}$. Then by Lemma~\ref{lem:dijkstra} we
have $P_e^{B+x}\cap(B-\Delta(B))=P_e^{B+z}\cap(B-\Delta(B))$ for every
$e\in B-\Delta(B)$, denote this common set by $H_e$. Fix an $e\in
B-\Delta(B)$ and an $i\in\{x,z\}$ and let the unique ordering of
$P_e^{B+i}$ be $(a_1,a_2,\ldots,a_k)$ according to
Theorem~\ref{thm:lfgpath}.  Then $i\in P_e^{B+i}$ is again obvious by
the definition of $\Delta(B)$, so let $i=a_j$ for some $1\le j\le
k$. Then $\{a_1,\ldots,a_j\}=P_i^{B+i}$ by Theorem~\ref{thm:lfgpath}.
Since $\Delta(B)+i\in\mathcal{F}$ is again true as in the proof of
Lemma~\ref{lem:dijkstra}, $P_i^{B+i}\subseteq \Delta(B)+i$ by the
definition of a path. Furthermore, if $y=a_t$ for some $j<t\le k$ then
$i\in P_y^{B+i}$ by Theorem~\ref{thm:lfgpath} and hence
$y\in\Delta(B)$ is impossible because that would imply $i\in
P_y^{B+i}\subseteq\Delta(B)$ by $\Delta(B)\in\mathcal{F}$. All these
together imply $P_e^{B+i}=P_i^{B+i}\cup H_e$.
Since $f(P_x^{B+x})\le f(P_z^{B+z})$ was shown above, this implies
$f(P_e^{B+x})\le f(P_e^{B+z})$ by property (ii) for every $e\in
B-\Delta(B)$. This completes the proof by (\ref{eq:dijkstra}).\qed
\end{proof}
\end{theopargself}

Since $f(P)=c(P)$ obviously fulfills the monotonicity constraints (i)
and (ii) for all non-negative valued weight functions
$c:S\rightarrow\mathbb{R}^+$, Theorem~\ref{thm:dijkstra} is indeed a
generalization Theorem~\ref{thm:boyd}.  Another application of
Theorem~\ref{thm:dijkstra} is to set $f(P)=\max\{c(x):x\in P\}$ for a
weight function $c:S\rightarrow\mathbb{R}$, which again obviously
fulfills conditions (i) and (ii). Theorem~\ref{thm:dijkstra} implies
the fact, which was also proved in \cite{Boyd}, that in local forest
greedoids the greedy algorithm finds a minimum base with respect to
$w$ in this case. If applied to the branching greedoid (and for
maximizing $(-w)$), this implies the well-known fact that the
corresponding modification of Dijkstra's algorithm solves the widest
path problem (also known as the bottleneck shortest path problem) in
graphs.


\end{document}